\documentclass[11pt]{amsart}
\usepackage[margin=1in]{geometry}
\usepackage{amssymb,amsmath}
\usepackage{hyperref}

\newtheorem{theorem}{Theorem}[section]
\newtheorem{lemma}[theorem]{Lemma}
\newtheorem{proposition}[theorem]{Proposition}

\newtheorem{mainthm}{Theorem}

\theoremstyle{definition}
\newtheorem{definition}[theorem]{Definition}
\newtheorem{remark}[theorem]{Remark}
\newtheorem{example}[theorem]{Example}
\newtheorem{question}[theorem]{Question}

\DeclareMathOperator{\Ann}{Ann}
\DeclareMathOperator{\Spec}{Spec}
\newcommand{\p}{\mathfrak p}
\newcommand{\q}{\mathfrak q}
\newcommand{\Z}{\mathbb Z}

\newcommand{\R}{\mathbb R}

\title[Semiprime and real submodules]{Semiprime and real submodules of arbitrary modules\\ are intersections of prime and real prime submodules}
\author{Jakob Cimpri\v c}
\address{University of Ljubljana, Faculty of Mathematics and Physics, Department of Mathematics, Jadranska 21, SI-1000 Ljubljana, Slovenia, and Institute of Mathematics, Physics and Mechanics, Jadranska 19, SI-1000 Ljubljana, Slovenia}
\email{jaka.cimpric@fmf.uni-lj.si}
\thanks{The author acknowledges the financial support from the Slovenian Research and Innovation Agency (ARIS), research core funding No.\ P1-0222 
and research projects No.\ J1-60011 
and No.\ J1-50002.} 


\begin{document}

\begin{abstract}
Let $R$ be a commutative ring. A submodule $N$ of an $R$-module $M$ is semiprime if $m\in(N:m)M$ implies $m\in N$. It is known that every semiprime submodule of $R^n$ is an intersection of prime submodules, and Aryapoor extended this to finitely generated modules. We show that no finiteness assumption is needed: every semiprime submodule of an \emph{arbitrary} $R$-module is the intersection of the prime submodules containing it. The proof is short and uses neither coordinates nor presentations; its key step is that a nonzero element $q$ of a semiprime module $Q$ satisfies $sq\notin\p Q$ for all $s\notin\p$, where $\p$ is any minimal prime over the annihilator of $q$. We define an $R$-module $M$ to be real if $\sum_i m_i\otimes m_i=0$ in $M\otimes_RM$ implies $m_i=0$ for all $i$, and a submodule $N\subseteq M$ real if the module $M/N$ is real. We prove the corresponding result for real submodules: every real submodule of an arbitrary module is the intersection of the real prime submodules containing it. 
On $R^n$ our notion of realness coincides with the coordinate definition used in earlier work on real Nullstellens\"atze for free modules.
\end{abstract}

\maketitle

\section{Introduction}

Throughout, $R$ is a commutative ring with $1$ and modules are unital. A submodule $P$ of an $R$-module $M$ is \emph{prime} if $P\neq M$ and $rm\in P$ implies $m\in P$ or $rM\subseteq P$ \cite{AR,Dauns,Lu}. Prime submodules generalize prime ideals, and a basic problem is to describe the intersection of all prime submodules containing a given submodule $N$ (which is $M$ if there is none), the \emph{radical} of $N$ (see \cite{Azizi,MM86,Sarac} and the references there; the convention for the empty intersection is that of \cite{MMS}). In \cite{MMS}, a submodule is even called semiprime if it is an intersection of prime submodules; Theorem~\ref{thmA} below shows that this agrees with Definition~\ref{def:semiprime} for all modules. The obvious generalization of semiprime ideals, $r^2m\in N\Rightarrow rm\in N$ \cite[Definition~1.9(ii)]{Dauns}, does not characterize intersections of prime submodules; see \cite[Theorem~11]{JS}: if $R$ is a unique factorization domain which is not a B\'ezout domain, for instance $R=\Z[x]$ or $R=k[x,y]$, then $R\oplus R$ has a submodule which satisfies this condition but is not an intersection of prime submodules. The right notion for submodules of $R^n$ was introduced in \cite{C22a}, where it was shown that every semiprime submodule of $R^n$ is an intersection of prime submodules; a satisfactory extension to general modules was left open there. Aryapoor \cite{Ary} proposed the following extension and proved the decomposition for finitely generated modules.

\begin{definition}[\cite{C22a,Ary}]\label{def:semiprime}
A submodule $N$ of $M$ is \emph{semiprime} if for every $m\in M$,
\[m\in(N:m)M\ \Longrightarrow\ m\in N,\qquad\text{where }(N:m):=\{r\in R:\ rm\in N\}.\]
\end{definition}

For $M=R^n$ this is the condition ``$f_1f,\dots,f_nf\in N\Rightarrow f\in N$'' of \cite{C22a}, see \cite[Proposition~2.1]{Ary}, and for $M=R$ it is the usual notion of a semiprime (= radical) ideal. Every prime submodule is semiprime, and intersections of semiprime submodules are semiprime (Lemma~\ref{lem:basic}). Our first result is that the converse holds without any finiteness assumption.

\begin{mainthm}\label{thmA}
Let $N$ be a submodule of an arbitrary $R$-module $M$. The following are equivalent:
\begin{enumerate}
\item[(i)] $N$ is semiprime;
\item[(ii)] $N=\bigcap_{\p\in\Spec R}K(N,\p)$, where $K(N,\p):=\{m\in M:\ sm\in N+\p M\text{ for some }s\notin\p\}$;
\item[(iii)] $N$ is an intersection of prime submodules of $M$.
\end{enumerate}
Consequently, the radical of any submodule $N$ is the smallest semiprime submodule of $M$ containing $N$.
\end{mainthm}

For $M=R^n$ this is \cite[Theorem~1]{C22a}, whose proof proceeds by induction on $n$, and for finitely generated $M$ it is \cite[Theorem~2.6]{Ary}, deduced from the free case by a presentation. The notation $K(N,\p)$ is taken from \cite[\ensuremath{\mathsection}2]{C22a}, where it denotes the smallest $\p$-prime submodule of $R^n$ containing $N$; the same description holds for arbitrary $M$ (Lemma~\ref{lem:Np}(b)). Note that $K(N,\p)$ is a submodule of $M$, not of a localization of $M$. Our proof avoids both induction and presentations. Theorem~\ref{thmA} answers the question raised in \cite[Remark~2.2]{Ary} whether \cite[Theorem~2.6]{Ary} holds for modules that are not finitely generated; there it was observed, using \cite[Corollary~13]{JS}, that the answer is positive over Dedekind domains. 

The second theme is real algebra. A ring $R$, or an ideal $I$, is \emph{real} if $\sum_ia_i^2=0$ (resp.\ $\sum_ia_i^2\in I$) implies $a_i=0$ (resp.\ $a_i\in I$) for all $i$; a classical theorem states that every real ideal is the intersection of the real prime ideals containing it \cite[Theorem~6.5]{Lam}. Real submodules of free modules $R^n$ were introduced in \cite{C13,C22b} in order to prove real Nullstellens\"atze for matrix polynomials. We propose the following definition for arbitrary modules.

\begin{definition}\label{def:real}
An $R$-module $M$ is \emph{real} if for all $m_1,\dots,m_k\in M$,
\[\textstyle\sum_im_i\otimes m_i=0\ \text{ in }M\otimes_RM\quad\Longrightarrow\quad m_i=0\ \text{ for all }i.\]
A submodule $N\subseteq M$ is \emph{real} if $M/N$ is a real module. A \emph{real prime submodule} is a prime submodule which is real, and $M$ is \emph{semireal} if it has a proper real submodule.
\end{definition}

For $M=R$ this is the notion of a real ideal, since $R/I\otimes_RR/I=R/I$. Real submodules are semiprime (Lemma~\ref{lem:realbasic}), and a semiprime submodule $N$ is real if and only if every $(N:m)$ is a real ideal (Proposition~\ref{prop:realann}).

\begin{mainthm}\label{thmB}
Let $N$ be a submodule of an arbitrary $R$-module $M$. The following are equivalent:
\begin{enumerate}
\item[(i)] $N$ is real;
\item[(ii)] $N=\bigcap\{K(N,\p):\ \p\text{ a real prime ideal of }R\}$;
\item[(iii)] $N$ is an intersection of real prime submodules of $M$.
\end{enumerate}
Consequently, the intersection of the real prime submodules containing a submodule $N$ (which is $M$ if there is none; note that $M$ itself is a real submodule, as $M/M=0$) is the smallest real submodule containing $N$; a module is semireal if and only if it has a real prime submodule; and every submodule maximal among the proper real submodules is prime.
\end{mainthm}

For $M=R$, Theorem~\ref{thmB} is the classical representation of real ideals. For $M=R^n$ we show in Section~\ref{sec:comparison} that Definition~\ref{def:real} coincides with the coordinate condition ``$\sum_if_{i,j}f_i\in N$ for all $j$ implies $f_i\in N$'' and with the tensor condition ``$\sum_if_i\otimes f_i\in R^n\otimes N+N\otimes R^n$ implies $f_i\in N$'', variants of which are used in \cite{C13,C22b} (Theorem~\ref{thmC}); thus Theorem~\ref{thmB} contains a decomposition of real submodules of free modules into real prime submodules.

The proofs of both theorems are based on Lemma~\ref{lem:key}: if $Q$ is a semiprime module, $0\neq q\in Q$ and $\p$ is a minimal prime over $\Ann_R(q)$, then $q\notin K(0,\p)$, i.e.\ $sq\notin\p Q$ for every $s\notin\p$. Applied to $Q=M/N$, it shows that every $m\notin N$ lies outside $K(N,\p)$ for a suitable prime $\p$. Also, $K(N,\p)$ is a prime submodule unless it equals $M$ (Lemma~\ref{lem:Np}), and it is real when $\p$ is real (Lemma~\ref{lem:Kreal}).

Section~\ref{sec:prelim} collects preliminaries, Sections~\ref{sec:A} and \ref{sec:B} contain the proofs of Theorems~\ref{thmA} and \ref{thmB} together with examples, Section~\ref{sec:comparison} compares Definition~\ref{def:real} with the definitions of \cite{C13,C22b}, and Section~\ref{sec:remarks} contains remarks and questions. In a subsequent paper \cite{Cnull} we show that over Jacobson rings every prime submodule of a finitely generated module is an intersection of maximal submodules, and we use this, together with Theorems~\ref{thmA} and \ref{thmB}, to prove Hilbert's Nullstellensatz and the real Nullstellensatz for submodules of finitely generated modules and for coherent sheaves.

\section{Preliminaries}\label{sec:prelim}

Let $R$ be a commutative ring with $1$.
An $R$-module $M$ is \emph{prime} if $M\neq0$ and $\Ann_R(m)=\Ann_R(M)$ for every $0\neq m\in M$ (this is equivalent to the definition in \cite{AR}; cf.\ \cite[Theorem~1.3(iii)]{Dauns}), and \emph{semiprime} if $m\in\Ann_R(m)M$ implies $m=0$. For every submodule $N$ of $M$ and every  $m\in M$ with image $\bar m\in M/N$ we have $(N:m)=\Ann_R(\bar m)$. As usual, prime submodules are proper and prime modules are nonzero, in accordance with the convention that prime ideals are proper; semiprime and real submodules, on the other hand, are not required to be proper, and semiprime and real modules are not required to be nonzero.

\begin{lemma}\label{lem:basic}
Let $N$ be a submodule of an $R$-module $M$.
\begin{enumerate}
\item[(a)] $N$ is a prime submodule of $M$ if and only if $M/N$ is a prime module; then $(N:M)=\Ann_R(M/N)$ is a prime ideal.
\item[(b)] $N$ is a semiprime submodule of $M$  if and only if $M/N$ is a semiprime module.
\item[(c)] Every prime submodule is semiprime.
\item[(d)] Every intersection of semiprime submodules is semiprime.
\item[(e)] Submodules of semiprime modules are semiprime, and nonzero submodules of prime modules are prime.
\item[(f)] Direct products, and hence direct sums, of semiprime modules are semiprime.
\end{enumerate}
\end{lemma}

Part (a) is essentially \cite[Theorem~1.3]{Dauns}, \cite[Theorem~1]{Lu}, \cite[Lemma~1]{JS} and \cite[Lemma~1.1]{MMS}, and parts (b)--(d) are contained in \cite[Propositions~2.2, 2.3 and 2.5]{Ary}; we include the short proofs for completeness.

\begin{proof}
(a) If $N$ is a prime submodule then for every $0 \ne \bar{m} \in M/N$ and every $r \in  \Ann_R(\bar m)$ we have $m \not\in N$ and $rm \in N$ which implies $rM \subseteq N$; i.e. $r \in \Ann_R(M/N)$. The reverse inclusion always holds; so $M/N$ is a prime module. Conversely, if $M/N$ is a prime module then for every $r \in R$ 
and $m \in M$ such that $rm \in N$ we have $r\bar{m}=0$ in $M/N$ which implies that either $\bar{m}=0$ 
or $r \in  \Ann_R(\bar m)=\Ann_R(M/N)$; i.e. either $m \in N$ or $rM \subseteq N$.

(b) For $m\in M$ put $I_m:=(N:m)=\Ann_R(\bar m)$; it depends only on $\bar m$. Since $I_m(M/N)$ is the image of $I_m M$ in $M/N$, we have
$\bar m\in I_m(M/N)$ if and only if $m\in I_mM+N.$
Hence $M/N$ is a semiprime module if and only if ($*$) $m\in I_mM+N$ implies $m\in N$, whereas $N$ is a semiprime submodule if and only if ($**$) $m\in I_mM$ implies $m\in N$. Clearly ($*$) implies ($**$). Conversely, assume ($**$) and let $m=m'+n$ with $m'\in I_mM$ and $n\in N$. Then $\bar m'=\bar m$, so $I_{m'}=I_m$ and $m'\in I_{m'}M$; by ($**$), $m'\in N$, hence $m\in N$.

(c) If $M$ is prime and $q\neq0$, then $\Ann_R(q)M=\Ann_R(M)M=0\not\ni q$; with (a) and (b) this shows that prime submodules are semiprime.

(d) If $N=\bigcap_iN_i$, then $(N:m)\subseteq(N_i:m)$, so $m\in(N:m)M$ implies $m\in N_i$ for every $i$.

(e) Let $L\subseteq M$. If $M$ is semiprime and $q\in L$ satisfies $q\in\Ann_R(q)L$, then $q\in\Ann_R(q)M$, so $q=0$. If $M$ is prime and $L\neq0$, then for $0\neq m\in L$ we have $\Ann_R(M)\subseteq\Ann_R(L)\subseteq\Ann_R(m)=\Ann_R(M)$, so $\Ann_R(m)=\Ann_R(L)$ for all $0\neq m\in L$.

(f) Let $M=\prod_\alpha M_\alpha$ with all $M_\alpha$ semiprime, and let $q=\sum_jb_jy_j$ with $b_j\in\Ann_R(q)$ and $y_j\in M$. For each $\alpha$, the $\alpha$-component $q_\alpha=\sum_jb_j(y_j)_\alpha$ lies in $\Ann_R(q)M_\alpha\subseteq\Ann_R(q_\alpha)M_\alpha$, so $q_\alpha=0$; hence $q=0$. Direct sums are submodules of direct products; apply (e).
\end{proof}

Proposition \ref{rem:dauns} compares Definition~\ref{def:semiprime} with Dauns' condition \cite[Definition~1.9(ii)]{Dauns} that all ideals $(N:m)$ are radical. Part (a) is \cite[Proposition~2.4]{Ary} and is used in the proof of Theorem~\ref{thmA}. Part (b) is not used there; although it follows from \cite[Corollary~13, (i)$\Rightarrow$(iii)]{JS}, we provide a direct proof to show how \cite[Corollary~13, (i)$\Rightarrow$(iii)]{JS} can be recovered from Theorem~\ref{thmA}.

\begin{proposition}[Dauns' condition]\label{rem:dauns}
Let $N$ be a submodule of an $R$-module $M$.
\begin{enumerate}
\item[(a)] If $N$ is semiprime, then $(N:m)$ is a radical ideal for every $m\in M$. In particular, if $M$ is semiprime, then $\Ann_R(q)$ is radical for every $q\in M$.
\item[(b)] If $R$ is a Dedekind domain and $(N:m)$ is radical for every $m\in M$, then $N$ is semiprime.
\end{enumerate}
\end{proposition}

\begin{proof}
By Lemma~\ref{lem:basic}(b), $(N:m)=\Ann_R(\bar m)$, and $N$ is semiprime if and only if $Q:=M/N$ is.

(a) If $r^2q=0$, then $r\in\Ann_R(rq)$, so $rq\in\Ann_R(rq)Q$, implying $rq=0$. Thus $\Ann_R(q)$ is radical. 

(b) Suppose that $0\neq q\in IQ$ with $I:=\Ann_R(q)$. Then $I\neq0$, for otherwise $IQ=0$. Let $\p\supseteq I$ be a maximal ideal. Then $q/1\neq0$ in $Q_\p$, since $sq=0$ would give $s\in I\subseteq\p$. Since $\p\neq0$, the local ring $R_\p$ is a discrete valuation ring; let $\pi$ be a uniformizer. The ideal $IR_\p$ is nonzero and contained in $\p R_\p$, so $IR_\p=(\pi^k)$ for some $k\ge1$; moreover $IR_\p$ is radical, because $I$ is radical and localization preserves radical ideals, whence $IR_\p=(\pi)$. Writing $q=\sum_ja_jy_j$ with $a_j\in I$ and $a_j/1=\pi c_j$, we get $q/1=\pi x$ with $x:=\sum_jc_jy_j/1$, and $\pi^2x=\pi\cdot q/1=0$ because $\pi\in IR_\p$. Write $x=y/s$ and $\pi=c/w$ with $s,w\notin\p$. Then $\pi^2x=c^2y/(w^2s)=0$ in $Q_\p$, so $uc^2y=0$ for some $u\notin\p$, and hence $(uc)^2y=u\cdot uc^2y=0$. Since $\Ann_R(y)$ is radical, $ucy=0$, and therefore $q/1=\pi x=cy/(ws)=ucy/(uws)=0$, a contradiction.
\end{proof}

Over $k[x,y]$ the two notions differ \cite[Theorem~11]{JS}.

For $R=\Z$, Proposition~\ref{rem:dauns} shows that an abelian group is semiprime if and only if every torsion element has squarefree order, since $n\Z$ is radical if and only if $n$ is $0$, $\pm1$, or squarefree. For instance, $\Z/6$ and $\Z\oplus\Z/2$ are semiprime, while $\Z/p^k$ with $k\ge2$ is not.

\begin{lemma}\label{lem:Np}
For every submodule $N$ of an $R$-module $M$ and every prime ideal $\p$ of $R$, the set
\[K(N,\p)=\{m\in M:\ sm\in N+\p M\text{ for some }s\notin\p\}\]
is a submodule of $M$ containing $N+\p M$. Moreover:
\begin{enumerate}
\item[(a)] Either $K(N,\p)=M$, or $K(N,\p)$ is a prime submodule with $(K(N,\p):M)=\p$.
\item[(b)] If $P\supseteq N$ is a prime submodule with $(P:M)=\p$, then $K(N,\p)\subseteq P$. Hence $K(N,\p)$ is the smallest prime submodule $P\supseteq N$ with $(P:M)=\p$ if such a submodule exists, and $K(N,\p)=M$ otherwise.
\end{enumerate}
\end{lemma}

\begin{proof}
Write $K:=K(N,\p)$. Taking $s=1$ shows $N+\p M\subseteq K$. If $m,m'\in K$, then $sm,s'm'\in N+\p M$ with $s,s'\notin\p$, so $ss'\notin\p$ and $ss'(am+bm')\in N+\p M$ for all $a,b\in R$; i.e.\ $am+bm'\in K$. Thus $K$ is a submodule.

(a) First, if $r\notin\p$ and $rm\in K$, say $srm\in N+\p M$ with $s\notin\p$, then $sr\notin\p$, so $m\in K$. Now let $K\neq M$. If $rm\in K$ and $m\notin K$, then $r\in\p$ by what we just saw, so $rM\subseteq\p M\subseteq K$; thus $K$ is prime. Finally $\p\subseteq(K:M)$ as $\p M\subseteq K$, and no $r\notin\p$ lies in $(K:M)$, since $rM\subseteq K$ would give $M\subseteq K$.

(b) If $m\in K(N,\p)$, then $sm\in N+\p M\subseteq P$ for some $s\notin\p$, since $\p M=(P:M)M\subseteq P$. As $P$ is prime and $s\notin(P:M)$, we get $m\in P$. The last statement follows from (a).
\end{proof}

Every proper ideal $I$ has a minimal prime over it, by Zorn's lemma applied to the primes containing $I$ ordered by reverse inclusion (an intersection of a chain of primes is prime).

\begin{lemma}\label{lem:minprime}
Let $I$ be a radical ideal and $\p$ a minimal prime over $I$.
\begin{enumerate}
\item[(a)] For every $a\in\p$ there is $t\notin\p$ with $ta\in I$.
\item[(b)] If $I$ is real, then $\p$ is real.
\end{enumerate}
\end{lemma}

\begin{proof}
(a) Put $A:=R/I$ and $P:=\p/I$. The primes of $A_P=(R/I)_\p$ correspond to the primes $\q$ of $R$ with $I\subseteq\q\subseteq\p$, so by minimality $A_P$ has exactly one prime, its maximal ideal $PA_P$. Write $\bar a$ for the image of $a\in R$ in $A$; the elements of $A_P$ are the fractions $\bar a/\bar s$ with $a\in R$ and $s\notin\p$. Since $I$ is radical, $A$ is reduced, and so is $A_P$: if $(\bar a/\bar s)^n=0$, then $\bar u\bar a^n=0$ for some $u\notin\p$, so $(\bar u\bar a)^n=0$, hence $\bar u\bar a=0$ and $\bar a/\bar s=0$. The nilradical of $A_P$ is the intersection of all primes of $A_P$, which is $PA_P$; since $A_P$ is reduced, the nilradical is zero, so $PA_P=0$. As $PA_P$ is the maximal ideal of $A_P$, it follows that $A_P$ is a field. For $a\in\p$ we get $\bar a/\bar1\in PA_P=0$, i.e.\ $\bar t\bar a=0$ for some $t\notin\p$, which means $ta\in I$.

(b) If $\sum_ia_i^2\in\p$, then by (a) there is $t\notin\p$ with $t\sum_ia_i^2\in I$, so $\sum_i(ta_i)^2\in I$, whence $ta_i\in I\subseteq\p$ and $a_i\in\p$.
\end{proof}

Both hypotheses in (a) are needed: for $R=k[x]$ and $\p=(x)$, neither $I=(x^2)$ (not radical) nor $I=0$ (for which $\p$ is not minimal) admits $t\notin\p$ with $tx\in I$. Part (b) is \cite[Lemma~2.9]{Lam}, applied to the real ring $R/I$.

\section{Proof of Theorem~\ref{thmA} and examples}\label{sec:A}

The proof of Theorem~\ref{thmA} rests on a single lemma. For a semiprime module $Q$ and $0\neq q\in Q$, it produces a prime ideal $\p$, namely any minimal prime over $\Ann_R(q)$, with $q\notin K(0,\p)$. Applied to $Q=M/N$, it separates every $m\notin N$ from $N$ by a submodule $K(N,\p)$, which is prime by Lemma~\ref{lem:Np}. After the proof we illustrate Theorem~\ref{thmA} by two examples.

\begin{lemma}[Key lemma]\label{lem:key}
Let $Q$ be a semiprime $R$-module, $0\neq q\in Q$, and $\p$ a minimal prime over $I:=\Ann_R(q)$. Then $sq\notin\p Q$ for every $s\notin\p$, i.e.\ $q\notin K(0,\p)$. 
\end{lemma}

\begin{proof}
$I$ is proper, and radical by Proposition~\ref{rem:dauns}(a). Suppose that $sq=\sum_ja_jy_j$ for some $s\notin\p$, $a_j\in\p$ and $y_j\in Q$. By Lemma~\ref{lem:minprime}(a) there are $t_j\notin\p$ with $t_ja_j\in I$; put $t:=\prod_jt_j$ and $u:=ts\notin\p$. Then
\[uq=\textstyle\sum_j(ta_j)y_j\in I\cdot Q\subseteq\Ann_R(uq)\cdot Q,\]
so $uq=0$ because $Q$ is semiprime. Thus $u\in I\subseteq\p$, a contradiction.
\end{proof}

\begin{proof}[Proof of Theorem~\ref{thmA}]
(i)$\Rightarrow$(ii): $N\subseteq K(N,\p)$ for all $\p$. Let $m\notin N$. Then $Q:=M/N$ is semiprime (Lemma~\ref{lem:basic}(b)) and $\bar m\neq0$. For a minimal prime $\p$ over $\Ann_R(\bar m)$, Lemma~\ref{lem:key} gives $s\bar m\notin\p Q$, i.e.\ $sm\notin N+\p M$, for every $s\notin\p$; that is, $m\notin K(N,\p)$.
(ii)$\Rightarrow$(iii): by Lemma~\ref{lem:Np}(a), each $K(N,\p)\neq M$ is prime, and the $K(N,\p)=M$ may be omitted.
(iii)$\Rightarrow$(i): Lemma~\ref{lem:basic}(c),(d).
The last statement follows because the semiprime submodules containing $N$ are closed under intersections, and each of them is an intersection of prime submodules containing $N$.
\end{proof}

The proof shows slightly more: in (ii) it suffices to take the primes $\p$ that are minimal over $(N:m)$ for some $m\notin N$.

\begin{example}[A decomposition]\label{ex:decomp}
Let $R=k[x,y]$, $M=R/(x)\oplus R/(y)$ and $N=0$. As a direct sum of semiprime modules, $M$ is semiprime (Lemma~\ref{lem:basic}(f)), but it is not prime, as $\Ann_R(\bar1,0)=(x)\neq(xy)=\Ann_R(M)$. For $\p=(x)$ we have $\p M=xM=0\oplus(x,y)/(y)$, and we claim $K(0,\p)=0\oplus R/(y)$. For $\supseteq$: if $v\in R/(y)$, then $s:=y\notin\p$ satisfies $s(0,v)=(0,yv)=0\in\p M$. For $\subseteq$: if $s(u,v)\in\p M$ with $s\notin\p$, then $su=0$ in $R/(x)\cong k[y]$, which is a domain in which $s\neq0$; hence $u=0$. Symmetrically $K(0,(y))=R/(x)\oplus0$, and $0=K(0,(x))\cap K(0,(y))$ is the decomposition of Theorem~\ref{thmA}.
\end{example}

\begin{example}[A decomposition in a free module of infinite rank]\label{ex:freeinf}
Let $R=k[x,y]$, $\mathfrak m=(x,y)$, $M=R^{(\mathbb N)}$ with basis $e_0,e_1,\dots$, and
\[N:=\{f\in M:\ \textstyle\sum_nf_n\in(x)\ \text{and}\ f_n\in\mathfrak m\ \text{for all }n\}.\]
Let $V:=k^{(\mathbb N_{\ge1})}$, on which $R$ acts through $R/\mathfrak m$. The map\[\varphi:M\to k[y]\oplus V,\qquad f\mapsto\bigl((\textstyle\sum_nf_n)(0,y),\ (f_n(0,0))_{n\ge1}\bigr),\]is $R$-linear and surjective with kernel $N$ (the two conditions also force $f_0\in\mathfrak m$), so $M/N\cong k[y]\oplus V$. This quotient is semiprime by Lemma~\ref{lem:basic}(f), since $k[y]\cong R/(x)$ and $V\cong(R/\mathfrak m)^{(\mathbb N_{\ge1})}$ are prime, hence semiprime, modules. It is not prime, since $\Ann_R(1,0)=(x)\neq\mathfrak m=\Ann_R(0,v)$ for $0\neq v\in V$. The minimal primes over the annihilators are $(x)$ and $\mathfrak m$, and
\[K(N,(x))=\{f:\ \textstyle\sum_nf_n\in(x)\},\qquad K(N,\mathfrak m)=\mathfrak mM:\]
for the first, note $N+xM=N$; if $\sum_nf_n\in(x)$, then $yf\in N$, since $\sum_nyf_n=y\sum_nf_n\in(x)$ and each $yf_n\in\mathfrak m$, and conversely $sf\in N$ with $s\notin(x)$ gives $s\sum_nf_n\in(x)$, hence $\sum_nf_n\in(x)$; for the second, note $N+\mathfrak mM=\mathfrak mM$, and every $s\notin\mathfrak m$ acts invertibly on $M/\mathfrak mM$. Hence $N=K(N,(x))\cap K(N,\mathfrak m)$ is the decomposition of Theorem~\ref{thmA}. Note that the quotient $M/N\cong k[y]\oplus V$ happens to be a direct sum of cyclic modules. Here $R$ is not a Dedekind domain, and $M$ and $N$ are not finitely generated, so none of \cite[Theorem~1]{C22a}, \cite[Corollary~13]{JS} and \cite[Theorem~2.6]{Ary} applies.
\end{example}

\section{Proof of Theorem~\ref{thmB} and examples}\label{sec:B}

We defined real modules and real submodules in Definition~\ref{def:real}. We first summarize which operations on modules and submodules preserve realness (Lemma~\ref{lem:realbasic}) and show that the submodules $K(N,\p)$ are real when $\p$ is a real prime ideal (Lemma~\ref{lem:Kreal}). We then compare realness of the submodule $N$ with realness of the ideals $(N:m)$, in analogy with Proposition~\ref{rem:dauns}, and prove Theorem~\ref{thmB} by the method of Section~\ref{sec:A}. The section ends with examples.

\begin{lemma}\label{lem:realbasic}
\begin{enumerate}
\item[(a)] Submodules of real modules are real.
\item[(b)] Direct products, and hence direct sums, of real modules are real.
\item[(c)] Every intersection of real submodules is real.
\item[(d)] Every real module is semiprime; hence every real submodule is semiprime.
\end{enumerate}
\end{lemma}

\begin{proof}
(a) Let $L$ be a submodule of a real module $M$, with inclusion $j:L\to M$. If $\sum_ix_i\otimes x_i=0$ in $L\otimes L$, applying $j\otimes j$ gives $\sum_ix_i\otimes x_i=0$ in $M\otimes M$, so $x_i=0$ for all $i$.

(b) Let $M=\prod_\alpha M_\alpha$ with all $M_\alpha$ real, and let $\pi_\alpha:M\to M_\alpha$ be the projections. If $\sum_ix_i\otimes x_i=0$ in $M\otimes M$, then applying $\pi_\alpha\otimes\pi_\alpha$ gives $\sum_i\pi_\alpha(x_i)\otimes\pi_\alpha(x_i)=0$ in $M_\alpha\otimes M_\alpha$, so $\pi_\alpha(x_i)=0$ for all $\alpha$, i.e.\ $x_i=0$. Direct sums are submodules of direct products; apply (a).

(c) If $N=\bigcap_\alpha N_\alpha$ with all $N_\alpha$ real, the canonical maps $M/N\to M/N_\alpha$ induce an embedding $M/N\hookrightarrow\prod_\alpha M/N_\alpha$; apply (b) and (a).

(d) If $q\in\Ann_R(q)M$, say $q=\sum_jb_jy_j$ with $b_jq=0$, then $q\otimes q=\sum_jy_j\otimes b_jq=0$, so $q=0$. Apply Lemma~\ref{lem:basic}(b).
\end{proof}

\begin{lemma}\label{lem:Kreal}
Let $\p$ be a real prime ideal and let $N$ be a submodule of an $R$-module $M$. Then $K(N,\p)$ is a real submodule of $M$; hence, if $K(N,\p)\neq M$, it is a real prime submodule with $(K(N,\p):M)=\p$.
\end{lemma}

\begin{proof}
Put $D:=R/\p$, $F:=\operatorname{Frac}(D)$ and $Q':=M/K(N,\p)$; write $\bar r$ for the image of $r\in R$ in $D$. Every element of $F$ has the form $\bar r/\bar s$ with $r\in R$, $s\notin\p$.

\emph{$F$ is formally real.} If $\sum_i(\bar a_i/\bar s)^2=0$, then $\sum_ia_i^2\in\p$, so all $a_i\in\p$ because $\p$ is real, i.e.\ all $\bar a_i/\bar s=0$.

\emph{$Q'$ embeds into an $F$-vector space.} If $K(N,\p)=M$, then $Q'=0$ is real. Otherwise $K(N,\p)$ is a prime submodule with $(K(N,\p):M)=\p$ by Lemma~\ref{lem:Np}(a), so $Q'$ is a torsion-free $D$-module by Lemma~\ref{lem:basic}(a) (cf.\ \cite[Lemma~1]{JS}). Hence $Q'$ embeds into its localization $V:=S^{-1}Q'$ at $S:=D\setminus\{0\}$, which is an $F$-vector space on which $R$ acts through $R\to D\subseteq F$.

\emph{$V$ is a real $R$-module.} We first show $V\otimes_RV=V\otimes_DV=V\otimes_FV$, via $v\otimes w\mapsto v\otimes w$. Since $\p V=0$, also $\p(V\otimes_RV)=0$, so $V\otimes_RV$ is a $D$-module. As $R\to D$ is surjective, a map $V\times V\to X$ into a $D$-module $X$ is $R$-bilinear if and only if it is $D$-bilinear; so $V\otimes_RV$ has the universal property of $V\otimes_DV$, i.e.\ $V\otimes_RV=V\otimes_DV$. Next, every element of $F$ has the form $ds^{-1}$ with $d\in D$ and $0\neq s\in D$, and for $v,w\in V$ we get in $V\otimes_DV$, using $D$-bilinearity for the scalar $s\in D$,
\[(s^{-1}v)\otimes w=(s^{-1}v)\otimes s(s^{-1}w)=s(s^{-1}v)\otimes(s^{-1}w)=v\otimes(s^{-1}w).\]
So the $D$-bilinear map $V\times V\to V\otimes_DV$ is $F$-bilinear, and $V\otimes_DV=V\otimes_FV$. If $\sum_iv_i\otimes v_i=0$, choose a basis $w_1,\dots,w_r$ of a finite-dimensional subspace containing all $v_i$ and write $v_i=\sum_jv_{ij}w_j$; the coefficient of $w_j\otimes w_j$ gives $\sum_iv_{ij}^2=0$ in $F$, so all $v_{ij}=0$ and all $v_i=0$.

Since $Q'$ is isomorphic to a submodule of $V$, it is real by Lemma~\ref{lem:realbasic}(a), i.e.\ $K(N,\p)$ is a real submodule. The last statement now follows from Lemma~\ref{lem:Np}(a).
\end{proof}

The next proposition is the real analogue of Proposition~\ref{rem:dauns}. Part (a) is used in the proof of Theorem~\ref{thmB}; parts (b) and (c) are not used there, but they show that realness can be tested on the ideals $(N:m)$.

\begin{proposition}\label{prop:realann}
Let $N$ be a submodule of an $R$-module $M$.
\begin{enumerate}
\item[(a)] If $N$ is real, then $(N:m)$ is a real ideal for every $m\in M$.
\item[(b)] If $N$ is semiprime and $(N:m)$ is a real ideal for every $m\in M$, then $N$ is real.
\item[(c)] If $R$ is a Dedekind domain and $(N:m)$ is a real ideal for every $m\in M$, then $N$ is real.
\end{enumerate}
\end{proposition}

\begin{proof}
We have $(N:m)=\Ann_R(\bar m)$, and $N$ is real if and only if $Q:=M/N$ is real.

(a) If $\sum_ia_i^2\in\Ann_R(q)$, then $\sum_i(a_iq)\otimes(a_iq)=\bigl((\sum_ia_i^2)q\bigr)\otimes q=0$, so $a_iq=0$ for all $i$.

(b) By Lemma~\ref{lem:basic}(b), $Q$ is semiprime. For $0\neq q\in Q$, let $\p$ be a minimal prime over $\Ann_R(q)$. Then $\p$ is real by Lemma~\ref{lem:minprime}(b), and $q\notin K(0,\p)$ by Lemma~\ref{lem:key}. Hence $0=\bigcap_{\p\text{ real}}K(0,\p)$ in $Q$, an intersection of real submodules (Lemma~\ref{lem:Kreal}), so $Q$ is real by Lemma~\ref{lem:realbasic}(c).

(c) Real ideals are radical, so $N$ is semiprime by Proposition~\ref{rem:dauns}(b); apply (b).
\end{proof}

If $N$ is prime, then $(N:m)=(N:M)$ for $m\notin N$ and $(N:m)=R$ for $m\in N$; so by (a) and (b), a prime submodule $N$ is real if and only if $(N:M)$ is a real ideal. For a semiprime submodule it is not enough that $(N:M)$ is real: for $M=\Z\oplus\Z/2$ and $N=0$ we have $(N:M)=0$, but $(N:(0,\bar1))=2\Z$ is not real.

\begin{example}[The converse of Proposition~\ref{prop:realann}(a) fails over $\R{[x,y]}$]\label{ex:JS}
This is the example of \cite[Theorem~11]{JS}. Let $R=\R[x,y]$, $J:=Rx+Ry\subseteq R$, $f:=(x,y)\in F:=R \oplus R$ and $N:=Jf=\{rf:\ r\in J\}$. 
We claim that $(N:m)\in\{0,J,R\}$ for every $m =(u,v)\in F$. Suppose there exists $0 \ne r \in (N:m)$. Then $ru=tx$ and $rv=ty$ for some $t\in J$, hence $r(uy-vx)=0$, and so $uy=vx$, as $r\neq0$.
It follows that $x\mid u$ (as $x\nmid y$), say $u=xs$, and $v=ys$, so $m=sf$ for some $s \in R$.
Then $rm\in N$ if and only if $rs\in J$ (as $R$ is a domain and $f\neq0$), so $(N:m)=(J:s)$, which is $J$ if $s\notin J$ and $R$ if $s\in J$.
This proves the claim, and all three ideals are real, as $R$ and $R/J\cong\R$ are real rings. To see that $N$ is not semiprime, we show that $f\in(N:f)F$ but $f\notin N$. Indeed, $(N:f)=J$ and $f=x(1,0)+y(0,1)\in JF=(N:f)F$, while $f\notin N$, since the coordinates of elements of $N$ lie in $J^2$. Hence $N$ is not semiprime, and therefore not real by Lemma~\ref{lem:realbasic}(d).
\end{example}

\begin{proof}[Proof of Theorem~\ref{thmB}]
(i)$\Rightarrow$(ii): Let $m\notin N$ and $Q:=M/N$. By Lemma~\ref{lem:realbasic}(d) and Proposition~\ref{prop:realann}(a), $Q$ is semiprime and $\Ann_R(\bar m)$ is a real, hence radical, ideal. Let $\p$ be a minimal prime over $\Ann_R(\bar m)$. Then $\p$ is real (Lemma~\ref{lem:minprime}(b)), and $m\notin K(N,\p)$ by Lemma~\ref{lem:key}.
(ii)$\Rightarrow$(iii): each $K(N,\p)\neq M$ is a real prime submodule by Lemma~\ref{lem:Kreal}.
(iii)$\Rightarrow$(i): Lemma~\ref{lem:realbasic}(c).

For the consequences, note first that the real submodules containing $N$ are closed under intersections (Lemma~\ref{lem:realbasic}(c)), and that each of them is an intersection of real prime submodules containing $N$. Next, if a proper submodule $N$ of $M$ is real, then (iii) provides a real prime submodule containing $N$; conversely, a real prime submodule is proper. Finally, if $N$ is maximal among the proper real submodules, every real prime submodule containing $N$ equals $N$, and there is at least one.
\end{proof}

\begin{example}[Real modules]\label{ex:cyclic}
(1) Since $R/I\otimes_RR/I=R/I$, the module $R/I$ is real if and only if $I$ is a real ideal. Thus $\Z/p$ is prime but not real, since $1^2+\dots+1^2$ ($p$ terms) lies in $p\Z$, and for the prime ideal $I=(x^2+y^2+1)$ of $\R[x,y]$ the module $\R[x,y]/I$ is prime but not real, as $-1$ is a sum of squares in it. So real modules form a proper subclass of semiprime modules.

(2) If $R$ is real, then every free module $R^{(I)}$ is real by Lemma~\ref{lem:realbasic}(b), hence so is every submodule of a free module, in particular every projective module (Lemma~\ref{lem:realbasic}(a)); for instance, every projective module over $\R[x_1,\dots,x_d]$ is real.

(3) If $R$ is a principal ideal domain, then an $R$-module $M$ is real if and only if $\Ann_R(q)$ is a real ideal for every $q\in M$, by Proposition~\ref{prop:realann}(a),(c). For $R=\Z$ the only real ideals are $0$ and $\Z$, so an abelian group is real if and only if it is torsion-free.
\end{example}

\begin{example}[A decomposition into real primes]\label{ex:realinf}
Let $R=\R[x,y]$ and $M=\bigoplus_{a\in\R}R/(x-a)$, which is not finitely generated. Each summand $R/(x-a)\cong\R[y]$ is real, so $M$ is real (Lemma~\ref{lem:realbasic}(b)). For $b\neq a$, the element $x-a$ acts on $R/(x-b)$ as the nonzero scalar $b-a$, so $(x-a)M=\bigoplus_{b\neq a}R/(x-b)$. We claim $K(0,(x-a))=(x-a)M$. For $\supseteq$ take $s=1$. For $\subseteq$, if $s\notin(x-a)$ and $sq\in(x-a)M$, then the $a$-component of $sq$ is $0$; as $s$ acts on $R/(x-a)\cong\R[y]$ as the nonzero polynomial $s(a,y)$, the $a$-component of $q$ is $0$. Hence $0=\bigcap_{a\in\R}K(0,(x-a))$ is the decomposition of Theorem~\ref{thmB}, into the real primes $K(0,(x-a))$ with $(K(0,(x-a)):M)=(x-a)$.
\end{example}

\section{Comparison with the coordinate definition on free modules}\label{sec:comparison}

For a submodule $N$ of $R^n$, the following coordinate condition is used in \cite{C22b} (see \cite{C13} for its relation to real left ideals of $M_n(R)$):
\[(\mathrm{C})\qquad f_1,\dots,f_k\in R^n,\ \ \textstyle\sum_if_{i,j}f_i\in N\ \text{ for }j=1,\dots,n\quad\Longrightarrow\quad f_i\in N\ \text{ for all }i,\]
where $f_{i,j}$ is the $j$-th coordinate of $f_i$. We show that it is equivalent to Definition~\ref{def:real}; the implication (C)$\Rightarrow$(real) uses the decomposition machinery.

\begin{lemma}\label{lem:tensorM}
Let $N$ be a submodule of an $R$-module $M$. Then $N$ is real if and only if $\sum_im_i\otimes m_i\in\operatorname{im}(M\otimes N+N\otimes M\to M\otimes M)$ implies $m_i\in N$ for all $i$.
\end{lemma}

\begin{proof}
Let $Q:=M/N$. Let $\pi:M\to Q$ be the quotient map, and let $L\subseteq M\otimes M$ be the submodule generated by all $m\otimes n$ and $n\otimes m$ with $m\in M$, $n\in N$; this is $\operatorname{im}(M\otimes N+N\otimes M\to M\otimes M)$. We claim that $\ker(\pi\otimes\pi)=L$. This is the case $\varphi=\psi=\pi$ of \cite[Theorem~2.19]{Conrad}, which describes the kernel of $\varphi\otimes\psi$ for surjective $\varphi,\psi$; for completeness we recall the argument. Clearly $\pi\otimes\pi$ is surjective and vanishes on $L$, so it induces $\alpha:(M\otimes M)/L\to Q\otimes Q$, $m\otimes m'+L\mapsto\bar m\otimes\bar m'$. Conversely, $(\bar m,\bar m')\mapsto m\otimes m'+L$ is well defined, since replacing $m,m'$ by $m+n,m'+n'$ with $n,n'\in N$ changes $m\otimes m'$ by $n\otimes m'+m\otimes n'+n\otimes n'\in L$; it is $R$-bilinear, so it induces $\beta:Q\otimes Q\to(M\otimes M)/L$ with $\beta(\bar m\otimes\bar m')=m\otimes m'+L$. As $\alpha\beta$ and $\beta\alpha$ are the identity on generators, $\alpha$ is an isomorphism and $\ker(\pi\otimes\pi)=L$. Since $(\pi\otimes\pi)(\sum_im_i\otimes m_i)=\sum_i\bar m_i\otimes\bar m_i$, we get $\sum_i\bar m_i\otimes\bar m_i=0$ if and only if $\sum_im_i\otimes m_i\in L$, so the two conditions in the lemma have the same hypothesis and the same conclusion. (If $M$ is flat, e.g.\ free, then $M\otimes N$ and $N\otimes M$ embed into $M\otimes M$ and $L=M\otimes N+N\otimes M$.)
\end{proof}

\begin{mainthm}\label{thmC}
For a submodule $N\subseteq R^n$ the following are equivalent:
\begin{enumerate}
\item[(i)] $N$ is real \textup{(}Definition~\ref{def:real}\textup{)};
\item[(ii)] $\sum_if_i\otimes f_i\in R^n\otimes N+N\otimes R^n$ implies $f_i\in N$ for all $i$;
\item[(iii)] $N$ satisfies \textup{(C)}.
\end{enumerate}
\end{mainthm}

\begin{proof}
Put $Q:=R^n/N$. (i)$\Leftrightarrow$(ii) is Lemma~\ref{lem:tensorM}, as $R^n$ is flat.

(i)$\Rightarrow$(iii): If $\sum_if_{i,j}\bar f_i=0$ in $Q$ for all $j$ where $f_i=\sum_jf_{i,j}e_j$,
\[\textstyle\sum_i\bar f_i\otimes\bar f_i=\sum_j\bar e_j\otimes\sum_if_{i,j}\bar f_i=0\quad\text{in }Q\otimes Q,\]
so $\bar f_i=0$ for all $i$.

(iii)$\Rightarrow$(i): First, $Q$ is semiprime: if $f\in(N:f)R^n$, then $f_j\in(N:f)$ for all $j$, i.e.\ $f_jf\in N$, which is (C) for the single vector $f$; so $f\in N$. Second, $(N:f)$ is real for every $f$: if $\sum_ka_k^2f\in N$, the vectors $a_kf$ satisfy $\sum_k(a_kf_j)(a_kf)=f_j\bigl(\sum_ka_k^2f\bigr)\in N$ for all $j$, so $a_kf\in N$ by (C). Now let $f\notin N$ and let $\p$ be a minimal prime over the real, hence radical, ideal $(N:f)=\Ann_R(\bar f)$. By Lemma~\ref{lem:minprime}(b), $\p$ is real, and by Lemma~\ref{lem:key}, $f\notin K(N,\p)$. So $N=\bigcap_{\p\text{ real}}K(N,\p)$ is an intersection of real submodules (Lemma~\ref{lem:Kreal}), hence real.
\end{proof}

By Theorem~\ref{thmC}, Theorem~\ref{thmB} applied to $M=R^n$ decomposes every submodule satisfying (C) into real prime submodules of $R^n$. The same argument, applied to a presentation $R^{(I)}\twoheadrightarrow M$, shows that for every module $M$ Definition~\ref{def:real} is equivalent to the coordinate condition on the preimage of $N$ in $R^{(I)}$.

\section{Remarks and questions}\label{sec:remarks}

\begin{remark}[Semiprime submodules over noncommutative rings]\label{rem:ncsemi}
For a noncommutative ring $R$ and a left $R$-module $M$, replace $(N:m)$ by $\langle N:m\rangle:=\{r\in R:\ rRm\subseteq N\}$, the annihilator of the submodule $R(m+N)$ of $M/N$, which is a two-sided ideal; call $N$ \emph{semiprime} if $m\in\langle N:m\rangle M$ implies $m\in N$ (the element annihilator $(N:m)$ would be too strong: for $R=M_2(k)$, $M=R$, $N=0$ and $a=E_{12}$ we have $a\in(0:a)R$, as $a^2=0$, although $0$ is semiprime in the classical sense $aRa\subseteq N\Rightarrow a\in N$). For $M=R^n$ this agrees with the definition in \cite{C22a} ``$m_iRm\subseteq N$ for $i=1,\dots,n$ implies $m\in N$'': as in \cite[Proposition~2.1]{Ary}, two-sidedness gives $\langle N:m\rangle R^n=\langle N:m\rangle^n$, so $m\in\langle N:m\rangle R^n$ if and only if $m_i\in\langle N:m\rangle$, i.e.\ $m_iRm\subseteq N$, for all $i$. For $n=1$ one recovers the classical semiprime one-sided ideals, $aRa\subseteq N\Rightarrow a\in N$ (see \cite[\ensuremath{\mathsection}1.1]{Dauns} and \cite[\ensuremath{\mathsection}1]{Sarac}), the semiprime analogue of Koh's one-sided ideals of prime type, $aRb\subseteq N\Rightarrow a\in N$ or $b\in N$ \cite{Koh}; prime submodules ($rRm\subseteq P\Rightarrow m\in P$ or $rM\subseteq P$ \cite{Dauns}) are semiprime in this sense. The semiprime submodules of \cite{Sarac} are of Dauns' type \cite[Definition~1.9(ii)]{Dauns}: $N\neq M$ and $rRrm\subseteq N$ implies $rm\in N$. Every proper semiprime submodule in our sense has this property (if $rRrm\subseteq N$, then $r\in\langle N:rm\rangle$, so $rm\in\langle N:rm\rangle M$), but not conversely, already over commutative rings \cite[Theorem~11]{JS}; for $M=R$ both notions give the semiprime left ideals.
\end{remark}

\begin{remark}[Real modules over $*$-rings]\label{rem:nc}
Let $R$ be a $*$-ring, i.e.\ a ring with involution $*$, and $M$ a left $R$-module. The \emph{conjugate module} $M^*$ is $M$ with the right action $m\cdot r:=r^*m$, so the tensor product $M^*\otimes_RM$ of a right and a left $R$-module is defined. For noncommutative $R$ it is only an abelian group: it is generated by the $m\otimes n$ subject to biadditivity and $(r^*m)\otimes n=m\otimes rn$. Call $M$ \emph{real} if $\sum_im_i\otimes m_i=0$ in $M^*\otimes_RM$ implies $m_i=0$ for all $i$; for commutative $R$ with the trivial involution this is Definition~\ref{def:real}.

For a left ideal $I$, the map $\bar a\otimes\bar b\mapsto a^*b+(I+I^*)$ is an isomorphism of abelian groups $(R/I)^*\otimes_R(R/I)\cong R/(I+I^*)$; here $I+I^*$ is only an additive subgroup of $R$, in general neither a left nor a right ideal. Hence $R/I$ is real if and only if $\sum_ia_i^*a_i\in I+I^*$ implies $a_i\in I$ for all $i$, which is the notion of a real left ideal used in \cite{CHKMN}.

Real modules are semiprime in the sense of Remark~\ref{rem:ncsemi}. Indeed, let $q\in JM$ with $J:=\Ann_R(Rq)$, say $q=\sum_jb_jy_j$ with $b_j\in J$. For $b\in J$ and $r\in R$, the element $z:=b^*rq=(rq)\cdot b$ satisfies $z\otimes z=(rq)\otimes bb^*rq=0$, so $z=0$; thus $J^*\subseteq J$. Therefore $q\otimes q=\sum_j(y_j\cdot b_j^*)\otimes q=\sum_jy_j\otimes b_j^*q=0$, and $q=0$. In particular, real left ideals are semiprime. The proofs of Theorems~\ref{thmA} and \ref{thmB} rely on localization at primes and do not extend directly.
\end{remark}

\begin{question}\label{q:nc}
For which noncommutative rings (with involution) is every semiprime, resp.\ real, submodule an intersection of prime, resp.\ real prime, submodules?
\end{question}

\begin{question}\label{q:inv}
Let $R$ be commutative with a nontrivial involution. Is every real submodule, in the sense of Remark~\ref{rem:nc}, an intersection of real prime submodules?
\end{question}

\section*{Acknowledgements}
During the preparation of this manuscript, the author utilized the AI assistants Claude Sonnet 5 and Claude Opus 5.5 (Anthropic) for conceptual discussions regarding proofs and literature, as well as for editing assistance. All arguments and references were independently verified by the author, who assumes full responsibility for the final content of this work.

\bibliographystyle{amsplain}

\end{document}